\documentclass[a4paper,10pt]{article}
\usepackage{multirow}
\usepackage{cite}
\usepackage{amsfonts}
\usepackage{hyperref}
\usepackage{amsmath}
\usepackage{lmodern}
\usepackage{bookmark}
\usepackage{abstract}
\usepackage{amssymb}
\usepackage[british]{babel}
\usepackage{amsthm}
\usepackage{graphicx}
\usepackage{subfigure}
\usepackage[all]{xy}
\usepackage{amsmath}
\usepackage{parskip}
\usepackage{tikz}
\usetikzlibrary{automata,arrows, matrix, positioning, decorations.markings}
\graphicspath{{chapter/}{figures/}}
\pgfdeclarelayer{back}
\pgfsetlayers{back,main} 
\tikzset{
  every matrix/.style = {
    matrix of nodes,
    text height = 1.5ex,
    text depth = 0.25ex,
  },
  column sep = 4pc,
  row sep = 3pc,
  > = angle 90,
  bend angle = 45,
  ampersand replacement = \#,
}
\tikzstyle{vecArrow} = [thick, decoration={markings,mark=at position
   1 with {\arrow[semithick]{open triangle 60}}},
   double distance=1.4pt, shorten >= 5.5pt,
   preaction = {decorate},
   postaction = {draw,line width=1.4pt, white,shorten >= 4.5pt}]
\tikzstyle{innerWhite} = [semithick, white,line width=1.4pt, shorten >= 4.5pt]
\makeatletter
\def\thm@space@setup{%
  \thm@preskip=\parskip \thm@postskip=0pt
}
\makeatother
\newtheorem{theorem}{Theorem}[section]
\newtheorem{lemma}[theorem]{Lemma}

\newtheorem{proposition}[theorem]{Proposition}
\newtheorem{corollary}[theorem]{Corollary}

\theoremstyle{definition}
\newtheorem{definition}[theorem]{Definition}

\theoremstyle{remark}
\newtheorem{remark}[theorem]{Remark}
\numberwithin{equation}{section}

\title{\vspace{-15mm}\fontsize{18pt}{10pt}\selectfont\textbf{Twisted $K$-homology, Geometric cycles and $T$-duality }} 
\author{
\large
\textsc{Bei Liu}\thanks{liu@uni-math.gwdg.de}\\[2mm] 
\normalsize Mathematisches Institut, Georg-August-Universit$\ddot{a}$ G$\ddot{o}$ettingen\\ 
\vspace{-5mm}
}

\begin{document}
\maketitle 

\begin{abstract}
Twisted $K$-homology corresponds to $D$-branes in string theory . In this paper we compare two different models of geometric twisted $K$-homology and get their equivalence. Moreover, we give another description of geometric twisted $K$-homology using bundle gerbes. In the last part we construct $T$-duality transformation for geometric twisted $K$-homology.
\end{abstract}

\section{Introduction}
String theory, as a candidate for quantum gravity, abstracts interests from both physicians and mathematicians. Its ultimate goal is to construct quantum theories of the basic structures of our universe. The starting point of string theory is that the fundamental units of our universe are 1-dimensional strings instead of point particles.String theorists find that there are five kinds of different string theories, i.e type I, type IIA, type IIB, heterotic $E_8$ and heterotic $SO(32)$ string theories, all of which are mathematically consistent (see \cite{String}). While there is only one universe, therefore it becomes important to study relations between different string theories. These relations are called duality. There are three kinds of duality in string theory: $S$-duality, $T$-duality and $U$-duality. $S$-duality is also called electric-magnetic duality \cite{KW}. $T$-duality, which is one of the main topic of this paper, is a duality exchanging winding number and momentum in the dynamic equation of $D$-branes. While $U$-duality can be seen as composition of $S$-duality and $T$-duality. $T$-duality provides an equivalence between type IIA and type IIB string theory. In physics, this equivalence is reflected by two important notions in string theory, $D$-branes and Ramond-Ramond fields living on $D$-branes. Briefly speaking, $D$-branes can be see as a submanifold of the spacetime manifold on which strings can end. A $Dp$-brane is a $p$-dimensional submanifold with some other structures on it.
\par
After Witten suggested that the Ramond-Ramond fields should be classified in (twisted) $K$-theory instead of de Rham cohomology in \cite{Witten}, twisted $K$-theory has been studied extensively (see \cite{Atiyah} and \cite{BCMMS}). There are mainly three different approaches to understand twisted $K$-theory:
\begin{enumerate}
\item bundle gerbes, which only works when twists are torsion cohomology classes
\item homotopy classes of sections of Fredlhom bundles
\item K-theory of the $C^{\ast}$-algebra of compact operator bundles determined by the twisting class
\end{enumerate}
The dual theory of twisted $K$-theory i.e twisted $K$-homology has also been studied as a mathematical interpretation of $D$-branes. The analytic approach to twisted $K$-homology is the same as the third for twisted $K$-theory we list above. However, geometric approaches are more useful for us to understand $D$-branes in string theory. In \cite{BCW} and \cite{BLW} they both give a construction of geometric twisted $K$-homology groups. We will show that they are equivalent to each other in this paper. After that, we will construct some properties of geometric twisted $K$-homology and then use the geometric twisted $K$-cycle in \cite{BLW} to construct the $T$-duality transformation for spacetime compactified over $S^1$.
\par
Now we give the structure of this paper. In the next section, we review different approaches to twisted $K$-homology. In section $3$ we show the equivalence of that two versions of geometric twisted $K$-homology. In section $4$ we establish some properties of geometric twisted $K$-homology. In section $5$ we discuss the charge map in \cite{BCW} and get a positive answer to the question on the charge map in the end of \cite{BCW}. In section $6$ we introduce another approach to define geometric twisted $K$-cycle using bundle gerbes.  In section $7$ we construct three transformations whose composition give $T$-duality transformation of geometric twisted $K$-homology. In the last section we will prove that the $T$-duality transformation is an isomorphism.

\textbf{Acknowledgement} The author thanks the Research Training Group 1493 "Mathematical Structures in Modern Quantum Physics" for the support during his Ph.D period. The author deeply thanks Thomas Schick for many helpful comments and discussions. Besides, the author also thanks Bailing Wang's wonderful lecture on geometric twisted $K$-homology and discussions.

\section{Reviews on twisted K-homology}
Let $X$ be a locally finite $CW$-complex and $\alpha: X\rightarrow K(\mathbb{Z}, 3)$ be a twist over $X$. Denote the projective unitary group over a separable Hilbert space $\mathbb{H}$ by $PU(\mathbb{H})$ and denote the $C^{\ast}$-algebra of compact operators by $\mathcal{K}$. Since the classifying space of $PU(\mathbb{H})$ is a model of $K(\mathbb{Z}, 3)$ (see \cite{Atiyah}), therefore $\alpha$ determines a principal $PU(\mathbb{H})$-bundle $\mathfrak{P}$ over $X$. We denote the associated $\mathcal{K}$-bundle of $\mathfrak{P}$ by $\mathcal{A}$. Then all of the continuous sections with compact support of $\mathcal{A}$ give rise to a $C^{\ast}$-algebra which we denote by $C^{\ast}(X, \alpha)$.  A direct way to construct a twisted $K$-homology theory for $(X, \alpha)$ is to use the $K$-homology of the continuous trace $C^{\ast}$-algebra of $C^{\ast}(X, \alpha)$. We denote this twisted $K$-homology group by $K^a_{\ast}(X, \alpha)$. However, it is very difficult to see the geometric meaning $K$-cycles through this approach. In \cite{BCW} and \cite{BLW}, more topological and geometric models are constructed.
Let $\mathfrak{K}$ be the complex $K$-theory spectrum and $\mathcal{P}_{\alpha}(\mathfrak{K})$the corresponding bundle of based spectra over $X$. In \cite{BLW}, the topological twisted $K$-homology group is defined to be
\begin{equation}\label{topological 1}
K^t_n(X, \alpha):= \lim_{k\rightarrow \infty}[S^{n+2k}, \mathcal{P}_{\alpha}(\Omega^{2k}\mathfrak{K})/X]
\end{equation}
This definition comes from the classical definition of homology theory by spectra, which is automatically a homology theory. In \cite{BLW}, B.L.\ Wang gave a geometric twisted $K$-homology. Before giving his constructions, we first review the definition of geometric cycles of $K$-homology in \cite{BHS}.
A geometry cycle on a pair of space $(X,Y)(Y \subset X)$ is a triple $(M,f,[E])$, such that\\
\begin{itemize}
\item $M$ is a compact spin$^c$-manifold(probably with boundary);
\item $f$ is a continuous map from $M$ to $X$ such that $f(\partial M)\subset Y$;
\item $[E]$ is a $K^0$-class of $M$
\end{itemize}
The definition of twisted cases is given in \cite{BLW} as follows.
\begin{definition}
A geometric cycle for $(X,Y,\alpha)$ is a quintuple $(M,\iota,\upsilon,\eta, [E])$ such that
\begin{itemize}
\item $M$ is a $\alpha$-twisted spin$^c$-manifold,i.e $M$ is a compact oriented manifold and the following diagram exists
\begin{equation}
\begin{tikzpicture}
               [
                    execute at begin node = \(,
                    execute at end node = \),
                    inner sep = .3333em,
                  ]
                  \matrix (m) {
                   |(i)|M \#  \#  |(l)| \textbf{BSO}  \\
                   |(j)|X \#  \#  |(k)| \ K(\mathbb{Z},3)  \\
                  };
                  \path[->]
                  (i)edge node[auto]{\iota} (j)
                  (i)edge node[auto]{\upsilon}(l)
                  (j)edge node[auto]{\alpha}(k)
                  (l)edge node[auto]{W_3}(k)
                  (l)edge node[auto]{\eta}(j);
                 \draw[vecArrow] (l) to (j);
                 \draw[innerWhite](l) to (j);
              \end{tikzpicture}
\end{equation}
Here $\upsilon$ and $W_3$ are classifying maps of the stable normal bundle of $M$ and the third integral Stiefel-Whitney class respectively, $\eta$ is a homotopy between $\alpha \circ \iota$ and $W_3 \circ \upsilon$. Moreover we require $\iota(\partial M)\subset Y$.
\item $[E]$ is an element class of $K^0(X)$ which is represented by a $\mathbb{Z}_2$-graded vector bundle $E$.
\end{itemize}
\end{definition}
Let $\Gamma(X, \alpha)$ be the collections of all geometric cycles for $(X, \alpha)$. To get geometric twisted $K$-homology, we still need an equivalence relation on $\Gamma(X, \alpha)$, which is generated by the following basic relations:
\begin{itemize}
  \item{\textbf{Direct sum - disjoint union}}
  If $(M,\iota,\upsilon,\eta, [E_1])$ and $(M,\iota,\upsilon,\eta, [E_2])$ are geometric cycles over $(X, \alpha)$, then
  \begin{equation}
  (M,\iota,\upsilon,\eta, [E_1])\cup (M,\iota,\upsilon,\eta, [E_2])\sim (M,\iota,\upsilon,\eta, [E_1]+[E_2])
  \end{equation}
  \item{\textbf{Bordism}}
  Given two geometric cycle $(M,\iota,\upsilon,\eta, [E_1])$ and $(M,\iota,\upsilon,\eta, [E_2])$ , if there exists a $\alpha$-twisted spin$^c$-manifold $(W, \iota, \upsilon, \eta)$ and $[E]\in K^0(W)$ such that
  \begin{equation}
  \delta(W, \iota, \upsilon, \eta)=-(M_1, \iota_1, \upsilon_1, \eta_1)\cup (M_2, \iota_2, \upsilon_2, \eta_2)
  \end{equation}
  and $\delta([E])= [E_1]\cup [E_2]$. Here $-(M_1, \iota_1, \upsilon_1, \eta_1)$ means the manifold $M_1$ with the opposite $\alpha$-twisted $Spin^c$ structure.
 \item{Spin$^c$ \textbf{vector bundle modification}}
  Given a geometric cycle $(M,\iota,\upsilon,\eta, E)$ and a spin$^c$ vector bundle $V$ over $M$ with even dimensional fibers, we can choose a Riemannian metric on $V\oplus \mathbb{R}$ and get the sphere bundle $\hat{M}= S(V\oplus \mathbb{R})$. Then the vertical tangent bundle $T^v(\hat{M})$ admits a natural spin$^c$ structure. Let $S^+_V$ be the associated positive spinor bundle and $\rho: \hat{M}\rightarrow M$ be the projection. Then
  \begin{equation}
  (M,\iota,\upsilon,\eta, E)\sim (\hat{M}, \iota\circ \rho, \upsilon\circ \rho, \eta\circ \rho, \rho^{\ast}E\otimes S^+_V).
  \end{equation}
  Here $\upsilon'$ is a classifying map of the stable normal bundle of $\hat{M}$ and $\eta'$ is a chosen homotopy between $W_3\circ \upsilon'$ and $\alpha\circ \iota\circ \rho$.
\end{itemize}
\begin{definition}\label{geometric 1}
$K^{g}_{\ast}:= \Gamma(X, \alpha)/ \sim$. Addition is given by disjoint union relation above. Let $K^{g}_0(X, \alpha)$(respectively $K^{g}_1(X, \alpha)$) be the subgroup of $K^{g}_{\ast}(X, \alpha)$ determined by all geometric cycles with even (respectively odd) dimensional $\alpha$-twisted spin$^c$-manifolds.
\end{definition}
There is an natural isomorphism $\mu$ between $K^{g}_{0/1}(X, \alpha)$ and $K^{a}_{0/1}(X, \alpha)$:
\begin{equation}
\mu(M, \iota, \upsilon, \eta, [E])= \iota_{\ast}\circ \eta_{\ast}\circ I^{\ast}\circ PD([E])
\end{equation}
Here $PD: K^i(M)\rightarrow K_{n+i}(M, W_3\circ \tau)$ is the Poincar$\check{e}$ duality map between $K$-group and $K$-homology group, $\iota_{\ast}$ is the push-forward map induced by $\iota$, $\eta_{\ast}$ is the canonical isomorphism induced by $\eta$
\begin{equation}
\eta_{\ast}: K^a_{0/1}(M, W_3\circ \upsilon)\rightarrow K^a_{1/0}(M, \alpha\circ \iota)
\end{equation}
 and $I: K^a_{0/1}(M, W_3\circ \tau)\rightarrow K^a_{0/1}(M, W_3\circ \upsilon)$ is a natural isomorphism induced by the anti-isomorphism of the associated $\mathcal{K}(\mathbb{H})$-bundles. The following theorem in \cite{BLW} states that $\mu$ is an isomorphism.
\begin{theorem}\label{Analytic index }
The assignment $(M, \iota, \upsilon, \eta, [E])\rightarrow \mu(M, \iota, \upsilon, \eta, [E])$, called the assembly map, defines a natural homomorphism
\begin{equation*}
\mu: K^g_{0/1}(X, \alpha)\rightarrow K^a_{0/1}(X, \alpha)
\end{equation*}
which is an isomorphism for any smooth manifold $X$ with a twisting $\alpha: X\rightarrow K(\mathbb{Z}, 3)$.
\end{theorem}
Before giving the definitions, we need to point out that there is another description of a twist over a space here. Let $X$ be a second countable locally compact Hausdorff topological space. Then a twist on $X$ in \cite{BEM} is a locally trivial bundle $\mathcal{A}$ of elementary $C^{\ast}$-algebras on $X$, i.e each fiber of $\mathcal{A}$ is an elementary $C^{\ast}$-algebra and with the structure group the automorphism group of $\mathcal{K}(\mathbb{H})$ for some Hilbert space $\mathbb{H}$. However, these two description can be transferred from one to the other.
In \cite{BCW} they used ($X$, $\mathcal{A}$) as the starting point and defined topological twisted $K$-cycles  as follows.
\begin{definition}
An $\mathcal{A}$-twisted $K$-cycles on $X$ is a triple $(M, \sigma, \psi)$ where
\begin{itemize}
  \item $M$ is a compact spin$^c$-manifold without boundary.
  \item $\phi: M\rightarrow X$ is a continuous map.
  \item $\sigma\in K_0(\Gamma(M, \phi^{\ast}\mathcal{A}^{op}))$
\end{itemize}
\end{definition}
Let $E(X, \alpha)$ be all of the $K$-cycles over $(X, \mathcal{A})$. Then the topological $\mathcal{A}$-twisted $K$-homology group over $(X, \mathcal{A})$ is defined by
\begin{equation}
K^{top}_{\ast}(X, \alpha):= E(X, \alpha)/\sim
\end{equation}
Here $\sim$ is a similar equivalence generated by disjoint unions, bordisom and vector bundle modification.
Moreover, a more geometric twisted $K$-cycle is given in \cite{BCW}, which is called $D$-cycle. In order to introduce $D$-cycle, we still also need the notion of spinor bundles.
\begin{definition}
A spinor bundle for a twisting $\mathcal{A}$ is a vector bundle $S$ of Hilbert spaces on $X$ together with a given isomorphism of twisting data
\begin{equation}
\mathcal{A}\cong \mathcal{K}(S)
\end{equation}
\end{definition}
One fact we need to know about spinor bundle is that a spinor bundle for $\mathcal{A}$ exists if and only if $DD(\mathcal{A})= 0$.
Now we can give the definition of $D$-cycle in \cite{BCW}.
\begin{definition}\label{geometric 2}
A $D$-cycle for $(X, \mathcal{A})$ is a $4$-tuple $(M, E, \phi, S)$ such that
\begin{itemize}
  \item $M$ is a closed oriented $C^{\infty}$-Riemannian manifold.
  \item $E$ is a complex vector bundle on $M$.
  \item $\phi$ is a continuous map from $M$ to $X$.
  \item $S$ is a spinor bundle for Cliff$^+_{\mathbb{C}}(TM)\otimes \phi^{\ast}\mathcal{A}^{op}$.
\end{itemize}
Two $D$-cycles $(M, E, \phi, S), (M', E', \phi', S')$ for $(X, \mathcal{A})$ are isomorphic if there is an orientation preserving isometric diffeomorphism $f: M\rightarrow M$ such that the diagram
\begin{equation}\label{isomorphism 2}
\begin{tikzpicture}
               [
                    execute at begin node = \(,
                    execute at end node = \),
                    inner sep = .3333em,
                  ]
                  \matrix (m) {
                   |(i)|M \#  \#  |(l)| M'  \\
                    \#  |(k)| \ X \# \\
                  };
                  \path[->]
                  (i)edge node[auto]{f} (l)
                  (i)edge node[below]{\phi}(k)
                  (l)edge node[below]{\phi'}(k);
\end{tikzpicture}
\end{equation}
commutes, and $f^{\ast}E'\cong E, f^{\ast}S'\cong S$.
\end{definition}
Let $\Gamma^D(X, \alpha)$ be the collection of all $D$-cycles over $(X, \mathcal{A})$. Similar, we can impose an equivalence generated by disjoint union, bordism and vector bundle modification on $\Gamma^D(X, \alpha)$ and get the geometric $\mathcal{A}$-twisted $K$-homology in \cite{BCW}, which we denote by $K^{geo}_{\ast}(X, \mathcal{A})$.
\begin{itemize}
  \item{\textbf{Direct sum - disjoint union}}
  If $(M, E_1, \iota, S)$ and $(M, E_2, \iota, S)$ are $D$-cycles for $(X, \alpha)$ then
  \begin{equation}
  (M, E_1, \iota, S)\cup (M, E_2, \iota, S)\sim (M, E_1\oplus E_2, \iota, S)
  \end{equation}
  \item{\textbf{Bordism}}
  Given two $D$ cycle $(M_0, E_0, \iota_0, S_0)$ and $(M_1, E_1, \iota_1, S_1)$ , if there exists a $4$-tuple $(W, E, \Phi, S)$ such that $W$ is a compact oriented Riemannian manifold with boundary, $E$ is a complex vector bundle over $W$ , $\Phi$ is a continuous map from $W$ to $X$ and
  \begin{equation}
  (\delta W, E|\delta W, \Phi|\delta W, S^+|\delta W)\cong (M_0, E_0, \iota_0, S_0)\cup (M_1, E_1, \iota_1, S_1)
  \end{equation}
  Here $S^+|\delta$ is the positive part of $S$. It is $S$ itself when $W$ is odd dimensional.
  \item{Spin$^c$-\textbf{vector bundle modification}}
  Given a $D$ cycle $(M, E, \iota, S)$ and a spin$^c$-vector bundle $V$ over $M$ with even dimensional fibers. Let $S'$ be the spinor bundle of Cliff$^+_{\mathbb{C}}(T\hat{M})\otimes (\phi\circ \rho)^{\ast}(\mathcal{A})^{op}$. Use the notation in the definition of geometric $K$-cycles. Then
  \begin{equation}
  (M, E, \iota, S)\sim (\hat{M}, \beta\otimes \rho^{\ast}E, \iota\circ \rho, \rho^{\ast}S ).
  \end{equation}
  Here $\beta$ is a complex vector bundle over $\hat{M}$ whose restriction to each fiber of $\rho$ is the Bott generator vector bundle of $\hat{M}$.  
\end{itemize}
In \cite{BCW} they gave a natural charge map $h: K^{geo}_{\ast}(X, \mathcal{A})\rightarrow K^{top}_{\ast}(X, \mathcal{A})$ as follows:
Let $(M, E, \phi, S)$ be a $D$-cycle and choose a normal bundle for $M$ with even dimensional fibers, then
\begin{equation}\label{charge map}
h(M, E, \phi, S):= (S(v\oplus \underline{\mathbb{R})}, \phi\circ \pi, \sigma)
\end{equation}
Here $\sigma$ is defined as the image of $E$ under the composition of $s_!$ and $\chi$.
 \begin{enumerate}
   \item Let $s: M\rightarrow S(\upsilon \oplus \underline{\mathbb{R}})$ be the canonical section of unite section on the trivial real line bundle. For simplicity, we denote the total space of sphere bundle $S(\upsilon \oplus \underline{\mathbb{R}})$ by $M$ and the bundle map by $\rho$. Then $s_!: K^0(M)\rightarrow K_0(\Gamma(\hat{M}, \rho^{\ast}(Cliff_{\mathbb{C}}(\upsilon))))$ is the Gysin homomorphism induced by $s$.
   \item $\chi: K_0(\Gamma(\hat{M}, \rho^{\ast}(Cliff_{\mathbb{C}}(\upsilon))))\rightarrow K_0(\Gamma(\hat{M}, (\phi\circ \rho)^{\ast}\mathcal{A}^{op}))$ is the isomorphism induced by the trivialisation of $TM\oplus \upsilon$ and the given spinor bundle $S$.
 \end{enumerate}
 They propose a question which asked that if $h$ is an isomorphism. We will prove in the paper that it is true when $X$ is a countable $CW$-complex.
\begin{remark}\label{FW}
We should note that all of these definitions of geometric twisted $K$-homology group cannot work for general twists. The twists should satisfy the Freed-Witten condition (see \cite{BLW}), i.e, there exists a manifold $M$ and a continuous map $l: M\rightarrow X$ s.t
\begin{equation}
l^{\ast}(H)+W_3(M)=0
\end{equation}
\end{remark}
\section{Equivalence between two versions of geometric twisted K-homology}
We list the following theorem from \cite{Plymen} first.
\begin{theorem}
Let $E$ be a vector bundle over a space $X$ and $\tilde{E}$ be the Clifford bundle of $E$. Let $W_3(E)$ be the third integral Stiefel-Whitney class of $E$. Then we have:
\begin{eqnarray}
W_3(E) =
\begin{cases}
DD(\tilde{E})   & if\ $E$\ has\ even\ dimension \\
DD(\tilde{E}^+)         & if\ $E$\ has\ odd\ dimension.
\end{cases}
\end{eqnarray}
\end{theorem}
\begin{remark}
Through the above theorem we can get a bit of idea why the two versions of geometric cycle are equivalent. The existence of spinor bundle implies the triviality of Dixmier-Douady class of $DD(Cliff^+(TM)\otimes \phi^{\ast}(\mathcal{A}^{op})$, so the choice of the spinor bundle $S$ for $Cliff^+(TM)\otimes \phi^{\ast}\mathcal{A}^{op}$ in the definition of $D$-cycle corresponds to the choice of the homotopy $\eta$ in the definition of geometric $K$-cycle in \cite{BLW}.
\end{remark}
\par
\noindent Let $X$ be a locally finite $CW$-complex and $\alpha: X\rightarrow K(\mathbb{Z}, 3)$ be a twist in the sense of \cite{BLW}. From the discussion in the beginning of section $2$ we know that $\alpha$ determines an associated $\mathcal{K}$-bundle $\mathcal{A}$ over $X$. While $\mathcal{A}$ is exactly a twist in the sense of \cite{BCW}. Therefore, we get the basic set up data of the two definitions of geometric twisted $K$-cycles are equivalent.
\par
\noindent
Moreover, we can consider twisting datas in different categories. Define $Twist_1(X)$ to be a category with continuous maps $\alpha: X\rightarrow K(\mathbb{Z}, 3)$ as objects. The morphisms of $Twist_1(X)$ are homotopies between objects. Define $Twist_2(X)$ to be a category with $\mathcal{H}$-bundles over $X$ as objects. The morphisms of $Twist_2(X)$ are isomorphisms of $\mathcal{K}$-bundles. Obviously, we can see that $Twist_1(X)$ and $Twist_2(X)$ are groupoids and the sets of equivalence classes for them are both   $H^3(X, \mathbb{Z})$.
Now we give two lemmas which is used later, whose proof are probably already known and so we don't claim that they are original.
\par
\begin{lemma}\label{K bundle}
Denote the projection from $X\times I$ to $X$ by $p$. For every $\mathcal{K}$-bundle $\mathcal{A}$ over $X\times I$, there exists a $\mathcal{K}$-bundle $\mathcal{A}'$ over $X$ such that $\mathcal{A}\cong p^{\ast}(\mathcal{A}')$.
\end{lemma}
The proof is obvious.
\begin{lemma}\label{canonical spinor bundle}
Let $\mathcal{A}$ be a $\mathcal{K}$-bundle over $X$, then there is a canonical spinor bundle for $\mathcal{A}\otimes \mathcal{A}^{op}$.
\end{lemma}
The proof can be found in \cite{BCW}.
\par
\noindent
Now we give a construction which is useful in the proof of the main theorem in this section. Assume $\mathcal{K}_1$ and $\mathcal{K}_2$ are two $\mathcal{K}$-bundles over $X$ and $\lambda: \mathcal{K}_1\rightarrow \mathcal{K}_2$ is an isomorphism. Then we can get a $\mathcal{K}$-bundle over $X\times [0, 1]$ as follows. First we can see that $\mathcal{K}_1\times [0, 1/2]$ and $\mathcal{K}_2\times [1/2, 1]$ are $\mathcal{K}$-bundle over $X\times[0, 1/2]$ and $X\times [1/2, 1]$ respectively. Then we can glue $\mathcal{K}_1\times [0, 1/2]$ and $\mathcal{K}_2\times [1/2, 1]$ at $\{1/2 \}\times X$ via $\lambda$ and get a $\mathcal{K}$-bundle $\mathcal{K}_0$ over $X\times [0,1]$.
Now we give the main theorem of this section.
\begin{theorem}
Let $X$ be a locally finite $CW$-complex. There exists an isomorphism $\mathcal{F}: K^g(X, \alpha)\rightarrow K^{geo}(X, \mathcal{A})$.
\end{theorem}
\noindent
To make the proof easier to read, we give the basic idea first. The idea is that we transform spinor bundles in $D$-cycles to $\mathcal{K}$-bundles over $M\times [0, 1]$. Then we use this $\mathcal{K}$-bundles to define homotopies in $\alpha$-twisted spin$^c$-manifolds and also define $F$ below. And we reverse the procedure to prove that $F$ is surjective. The injectivity of $F$ is essentially implied by the compatibility of $F$ with $\sim$. Now we start the proof.
\begin{proof}
Let $[x]$ be a class in $K^{geo}(X, \mathcal{A})$ and $(M^g, E, \phi, S)$ be a $D$-cycle which represents $[x]$. By definition $S$ is a spinor bundle of Cliff$^+_{\mathbb{C}(TM)}\otimes \phi^{\ast}\mathcal{A}^{op}$. And we denote the chosen isomorphism between $\mathcal{K}(S)$ and Cliff$^+_{\mathbb{C}}(TM)\otimes \phi^{\ast}\mathcal{A}^{op}$ by $\lambda$. We define $F([x])$ in $K^g(X, \alpha)$ to be a class represented by $(M, \phi, \upsilon, \eta, E)$, in which
\begin{itemize}
\item $M$ is the manifold of $M^g$ forgetting Riemiannian structure, $\phi$ and $E$ are the same map and bundle in the $D$-cycle;
\item $\upsilon$ is the classifying map of the stable normal bundle of $M$;
\item $\eta$ is a homotopy between $W_3\circ \upsilon$ and $\alpha\circ \phi$.
\end{itemize}
We only need to explain how $\eta$ is constructed from $(M^g, E, \phi, S)$. By Lemma \ref{canonical spinor bundle} we know that there exists a canonical Hilbert bundle $V$ over $M$ and a canonical isomorphism $c$ between $\mathcal{K}(V)\cong \phi^{\ast}\mathcal{A}\otimes \phi^{\ast}\mathcal{A}^{op}$. Combine $c$ and $h$ we get an isomorphism between $\mathcal{K}(S)\otimes \mathcal{A}$ and Cliff$^+_{\mathbb{C}}(TM)\otimes \mathcal{K}(V)$. According to the discussion before the theorem we can obtain a $\mathcal{K}$-bundle $\mathcal{W}$ over $M\times I$ such that $\mathcal{W}_{M\times \{0\}}\cong \mathcal{K}(S)\otimes \mathcal{A}$ and $\mathcal{W}_{M\times \{1\}}\cong Cliff^+_{\mathbb{C}}(TM)\otimes \mathcal{K}(V)$. Since $BPU(\mathbb{H})$ is also a classifying space of $\mathcal{K}$-bundles, therefore we can get that there exists an $\eta: X\times [0, 1]\rightarrow K(\mathbb{Z}, 3)$ such that $(\eta\circ (\phi\times Id))^{\ast}(\mathfrak{K})$ is isomorphic to $\mathcal{W}$. Moreover, we get two maps $\alpha_0$, $\alpha_1: X\rightarrow K(\mathbb{Z}, 3)$ by restricting $\eta$ to $X\times \{0\}$ and $X\times \{1\}$ respectively, which gives the following isomorphisms:
\begin{equation}
(\alpha_0\circ \phi)^{\ast}(\mathfrak{K})\cong \mathcal{K}(S)\otimes \mathcal{A}, \ \ (\alpha_1\circ \phi)^{\ast}(\mathfrak{K})\cong Cliff^+_{\mathbb{C}}(TM)\otimes \mathcal{K}(V)
\end{equation}
Since different choices of $\eta$ are homotopic to each other, so they represent the same class in $K^g(X, \alpha)$ via bordism relation. To show that $F$ is well defined, we still need to check that it is compatible with the relations which define $\sim$.
\begin{itemize}
\item From the definition of $F$ we can see
\begin{eqnarray*}
F([(M^g, E_1, \phi, S)]\cup [(M^g, E_2, \phi, S)])= [(M, E_1\oplus E_2, \phi, S)];\\
F([M^g, E_1, \phi, S])\cup F([M^g, E_2, \phi, S])= [(M, E_1\oplus E_2, \phi, S)]
\end{eqnarray*}
i.e $F$ respects the disjoint union relation.
\item Let $(M^g, E, \phi, S)$ be a bordism between $(M^g_0, E_0, \phi_0, S_0)$ and $(M^g_1, E_1, \phi_1,\\
S_1)$. Denote the associated isomorphisms of the spinor bundles by $h$, $h_0$ and $h_1$. Denote the chosen representing cycles of the image of their homology classes under $F$ by $(M, \phi, \upsilon, \eta, E)$, $(M_0, \phi_0, \upsilon_0, \eta_0, E_0)$ and $(M_1, \phi_1, \upsilon_1, \\
\eta_1, E_1)$ respectively. Choosing a tubular neighborhood of $M_0$ in $M$ we can get that $TM|_{M_0}\cong TM_0\otimes \underline{\mathbb{R}}$, so we get the stable normal bundle of $M_0$ agrees with the restriction of the stable normal bundle of $M$ to $M_0$ i.e $\upsilon|_{M_0}$ is homotopic to $\upsilon_0$. Similarly we can get $\upsilon_{M_1}$ is homotopic to $\upsilon_1$.
\par
  Let $\mathcal{W}$, $\mathcal{W}_0$ and $\mathcal{W}_1$ be the three $\mathcal{K}$-bundles over $M\times [0,1]$, $M_0\times [0, 1]$ and $M_1\times [0, 1]$ respectively. The isomorphism $c$ in Lemma \ref{canonical spinor bundle}  is  natural  and $h|_{\mathcal{K}(S_i)}= h_i$ ($i=0$, $1$), so we get that $\mathcal{W}|_{M_i\times[0, 1]}$  is isomorphic to $\mathcal{W}_i$ ($i=0$, $1$). This implies that $\eta|_{M\times \{i\}}$  is homotopic to $\eta_i$ ($i=0$, $1$). So $F$ is compatible with bordism.
\item  Use the notion in Section  $2.1$ and denote  $F([(\hat{M}^g, S^+_V\otimes\rho^{\ast}E, \phi\circ \rho, \rho^{\ast}S)$ by $[\hat{M}, S^+_V\otimes \rho^{\ast}E, \phi\circ \rho, \upsilon', \eta']$.  We only need to prove that $\eta'$ is homotopic the $\eta\circ \pi$. The is implied by the observation that the corresponding $\mathcal{K}$-bundles $\mathcal{V}'$ and $\mathcal{V}''$ over $\hat{M}\times [0, 1]$ are isomorphic.
\end{itemize}
Obviously $F$ maps a trivial $D$-cycle to a trivial geometric twisted $K$-cycle, so we get that $F$ is a well defined injective homomorphism. The left is to show that $F$ is surjective. For any class $[y]\in K^g(X, \alpha)$, choose a geometric cycle $(M, \phi, \upsilon, \eta, E)$ to represent it. $\eta$ induces a $\mathcal{K}$-bundle over $M\times [0, 1]$, which we denote by $\mathcal{X}$. By the definition of $\eta$ we have
$\mathcal{X}|_{M\times \{0\}}\cong$ Cliff$^+_{\mathbb{C}}(TM)\otimes \mathcal{K}$,   $\mathcal{X}|_{M\times \{1\}}\cong (\alpha\circ \phi)^{\ast}(\mathfrak{K})$.
Let $l$ be an isomorphism between Cliff$^+_{\mathbb{C}}(TM)\otimes \mathcal{K}$ and $(\alpha\circ \phi)^{\ast}(\mathfrak{K})$ and $\mathcal{L}$ be a $\mathcal{K}$-bundle constructed via gluing Cliff$^+_{\mathbb{C}}(TM)\otimes \mathcal{K}\times [0, 1/2]$ and $(\alpha\circ \phi)^{\ast}(\mathfrak{K}) \times [1/2, 1]$ by $l$ at $M\times \{1/2\}$. Then we have that $\mathcal{L}$ is isomorphic $\mathcal{X}$ since they have the same Dixmier-Dourady class. Combining $l$ with the canonical isomorphism in Lemma \ref{canonical spinor bundle}, we get a spinor bundle $S$ for Cliff$^+_{\mathbb{C}}(TM)\otimes \phi^{\ast}\mathcal{A}^{op}$. The $D$-cycle $(M^g, E, \phi, S)$ satisfies that $F([M^g, E, \phi, S])=[y]$.
\end{proof}
\begin{remark}
The above proposition shows that geometric $K$-cycles in \cite{BCW} and $D$-cycles in \cite{BLW} are equivalent to each other. Therefore we can choose one of them to construct the topological $T$-duality for geometric twisted $K$-homology. Moreover, $K^g(X, \alpha)\cong K^a(X, \alpha)$ imples that $K^{geo}(X, \alpha)$ is also isomorphic to $K^a(X, \alpha)$.
\end{remark}
\section{Properties of geometric twisted K-homology}
In this section we prove that geometric twisted $K$-homology satisfies Elienberg-Steenrod axioms of homology theory.
We first give a simple lemma.
\begin{lemma}\label{induced map}
If $f: Y\rightarrow X$ is a continuous map, then $f$ induces a homomorphism $f_{\ast}: K^g(Y, \alpha\circ f)\rightarrow K^g(X, \alpha)$.
\end{lemma}
\begin{proof}
Given a twisted geometric $K$-cycle $(M, \phi, \upsilon, \eta, E)$ over $(Y, \alpha\circ f)$, we define $f_{\ast}$ by
\begin{equation}
f_{\ast}(M, \phi, \upsilon, \eta, E)= (M, f\circ \phi, \upsilon, \eta, E)
\end{equation}
We need to check that $f_{\ast}$ is compatible with disjoint union, bordism and spin$^c$-vector bundle modification.
\begin{itemize}
 \item Given two geometric $K$-cycles $(M, \phi, \upsilon, \eta, E_i)$ ($i=1, 2$), we have
  \begin{equation*}
  f_{\ast}((M, \phi, \upsilon, \eta, E_1)\cup (M, \phi, \upsilon, \eta, E_2))=(M, \phi\circ f, \upsilon, \eta, E_1\oplus E_2)
  \end{equation*}
\item If $(M, \phi, \upsilon, \eta, E)$ is a bordism between $(M_1, \phi_1, \upsilon_1, \eta_1, E_1)$ and $(M_2, \phi_2, \upsilon_2, \\
  \eta_2, E_2)$, then clearly $(M, f\circ\phi, \upsilon, \eta, E)$ gives a bordism between $(M_1, f\circ \phi_1, \upsilon_1, \eta_1, E_1)$ and $(M_2, f\circ\phi_2, \upsilon_2, \eta_2, E_2)$.
\item Since $f((\hat{M}, \phi\circ \rho, \upsilon\circ \rho, \eta\circ (\rho\times Id), \rho^{\ast}E\otimes S^+_V))$ is $(\hat{M}, f\circ \phi\circ \rho, \upsilon\circ \rho, \eta\circ (\rho\times Id), \rho^{\ast}E\otimes S^+_V)$, which is exactly a spin$^c$-vector bundle modification of $(M, f\circ \phi, \upsilon, \eta, E)$, so we get that $f_{\ast}$ respects the spin$^c$-vector bundle modification relation.
\end{itemize}
\end{proof}
\begin{theorem}[\textbf{Homotopy}]\label{homotopy}
If $f: Y\rightarrow X$ is a homotopy equivalence, then the induced map $f_{\ast}: K^g_{\ast}(Y, \alpha\circ f)\rightarrow K^g_{\ast}(X, \alpha)$ is an isomorphism.
\end{theorem}
\begin{proof}
We first show that if $g: Y\rightarrow X$ is homotopic to $f$, then $K^g_{\ast}(Y, \alpha\circ f)\cong K^g_{\ast}(Y, \alpha\circ g)$. Let $H: Y\times [0, 1]\rightarrow X$ be a homotopy from $f$ to $g$ i.e $H(y, 0)=f(y)$ and $H(y, 1)= g(y)$. Given a twisted geometric $K$-cycle $(M, \phi, \upsilon, \eta, E)$ over $(Y, \alpha\circ f)$, we get a twisted geometric $K$-cycle over $(Y, \alpha\circ g)$ as follows: $(M, \phi, \upsilon, \eta', E)$. Here $\eta': M\times [0, 1]\rightarrow K(\mathbb{Z}, 3)$ is defined by
\begin{eqnarray*}
\eta'(m, t)=
\begin{cases}
\eta (m, 2t) & 0\leq t\leq 1/2  \\
\alpha \circ H(m, 2t-1)\circ (\phi\times Id) & 1/2\leq t\leq 1
\end{cases}
\end{eqnarray*}
It is not hard to check that the above map is compatible with the disjoint union, bordism and spin$^c$-vector bundle modification. We skip the details here since they are similar to the proof of the above lemma. Therefore we get a homomorphism  $H_{\ast}$ from $K^g_{\ast}(Y, \alpha\circ f)$ to $K^g_{\ast}(Y, \alpha\circ g)$. Similarly we can get the inverse of $H_{\ast}$ by using $H(1-t, y)$ as a homotopy from $g$ to $f$. So we get that $H_{\ast}$ is an isomorphism.
Clearly, we have that $f_{\ast}= g_{\ast}\circ H_{\ast}$.
Let $q: X\rightarrow Y$ be a homotopy inverse of $f: Y\rightarrow X$ i.e $f\circ q$ is homotopic to $id_X$ and $q\circ f$ is homotopic to $id_Y$. Denote the associated homotopies by $H_1$ and $H_2$ respectively. Then we have that $(f\circ q)_{\ast}= (H_1)_{\ast}$ and $(q\circ f)_{\ast}= (H_2)_{\ast}$. Since $(H_1)_{\ast}$ and $(H_2)_{\ast}$ are both isomorphisms, we get that $f_{\ast}$ is an isomorphism as well.
\end{proof}
\begin{theorem}[\textbf{Excision}]\label{Excision}
Let $(X, Y)$ be a pair of locally finite $CW$-spaces, $U$ is an open set of $X$ such that $\overline{U}\in Y$. Then we the inclusion
\begin{equation*}
 i: (X-U, Y-U)\hookrightarrow (X, Y)
\end{equation*}
induces an isomorphism
\begin{equation}
K^g_{\ast}(X-U, Y-U; \alpha\circ i)\cong K^g_{\ast}(X, Y; \alpha)
\end{equation}
\end{theorem}
\begin{proof}
By Lemma \ref{induced map} the inclusion $i$ induces a homomorphism $i_{\ast}$. And moreover $i_{\ast}$ is injective by its definition. The remainder is to prove that $i_{\ast}$ is surjective. For any $y\in K^g_{\ast}(X, Y; \alpha)$, we choose a geometric cycle $(M, \phi, \upsilon, \eta, E)$ to represent it. We choose a point $p_0\in U$. By the Urysohn's Lemma, there exists a continuous function $f: X\rightarrow [0, 1]$ such that $f(p_0)=0$ and $f(x)=1$ for any $x\in X-U$. Then $f\circ \phi$ is a continuous function from $M$ to $[0, 1]$. Let $W$ be a sub-manifold of $M\times [0, 1]$ defined by $W=\{(t, (f\circ \phi)^{-1}([t, 1])| t\in [0, 1]\}$. Let $\pi: W\rightarrow M$ be the canonical projection. For each slice $W_t= (f\circ \phi)^{-1}([t, 1])$ in $W$, we denote the inclusion of $W_t$ to $M$ by $i_t$. Define $\eta': W\times [0, 1]$ as follows
\begin{equation*}
\eta'(m,s, t)= \eta\circ (i_s\times id)(m, t)
\end{equation*}
Here $(m, s)\in W$ and $t\in [0, 1]$
Then the $\alpha$-twisted spin$^c$-manifold $(W, \phi\circ \pi, \upsilon\circ \pi, \eta')$, which gives a bordism between $(M, \phi, \upsilon, \eta, E)$ and $(f^{-1}(X-U), \phi\circ j, \upsilon\circ j, \eta\circ (j\times Id), j^{\ast}E)$. Here $j: f^{-1}(X-U)\hookrightarrow M$ is the inclusion. While $(f^{-1}(X-U), \phi\circ j, \upsilon\circ j, \eta\circ (j\times Id), j^{\ast}E)$ is also a twisted geometric $K$-cycle over $(X-U, Y-U; \alpha\circ i)$, whose image under $i_{\ast}$ is equivalent to $(M, \phi, \upsilon, \eta, E)$.
\end{proof}
\begin{remark}
In the above proof we assume each slice $W_t$ to be a sub-manifold of $M$, this is not true in general. But here we assume that both $X$ and $U$ are homotopic to finite $CW$-complexes, so we can always find a smooth manifold which is homotopy equivalent to $W_t$. This makes the proof still works for general cases.
\end{remark}
Another important axiom in Eilenberg-Steenrod axioms is the long exact sequence. Before moving on to the long exact sequence of geometric twisted $K$-homology, we first introduce a notion.
\begin{definition}
A twist $\alpha: X\rightarrow K(\mathbb{Z}, 3)$ is called representable if there exists an oriented real vector bundle $V$ over $X$ such that $W_3(V)= [\alpha]$. Here $[\alpha]$ is the pullback of the generator of $H^3(K(\mathbb{Z}), 3)$ along $\alpha$.
\end{definition}
\begin{theorem}[\textbf{Six-term exact sequence}]\label{six term exact sequence}
Let $Y$ be a sub-space of $X$ and $i$ be the inclusion map from $Y$ to $X$. Let $\alpha$ be a representable twist over $X$. Then we have a six-term exact sequence:
\begin{equation*}
\begin{tikzpicture}
[shorten >=1pt,node distance=1.5cm,auto]
\node []  (X_0)  {$K^g_0(Y, \alpha\circ i)$};
\node []  (X_1)  [right=of X_0] {$K^g_0(X, \alpha)$};
\node[]  (X_2)[right=of X_1] {$K^g_0(X, Y; \alpha)$};
\node[]  (Y_0)[below=of X_0] {$K^g_1(X, Y; \alpha)$};
\node[]  (Y_1) [right=of Y_0] {$K^g_1(X, \alpha)$};
\node[]  (Y_2) [right=of Y_1]  {$K^g_1(Y, \alpha\circ i)$};
\path[->]
(X_0) edge[above] node {$i_{\ast}$} (X_1)
(X_1) edge[above] node {$j_{\ast}$}  (X_2)
(X_2) edge[right] node {$\delta$} (Y_2)
(Y_2) edge[below] node {$i_{\ast}$} (Y_1)
(Y_1) edge[below] node {$j_{\ast}$}  (Y_0)
(Y_0) edge[left]  node {$\delta$} (X_0);
\end{tikzpicture}
\end{equation*}
Here the boundary operator is given by
\begin{equation}
\delta([M, \phi, \upsilon, \eta, E])= [(\partial M, \phi|_{\partial M}, \upsilon|_{\partial M}, \eta|_{\partial M\times [0, 1]}, E|_{\partial M})]
\end{equation}
\end{theorem}
To prove this theorem, we first prove two lemmas as a preparation.
\begin{lemma}\label{vector bundle modification}
Let $\theta= (M, \iota, \upsilon, \eta, E)$ be a geometric $K$-cycle over $(X, \alpha)$ and $E_i$ ($i=1$, $2$) be spin$^c$-vector bundles over $M$ with even dimensional fibers. Denote the vector bundle modification of $\theta$ with a spin$^c$-vector bundle $F$ by $\theta_F$. Then we have that $\theta_{E_1\oplus E_2}$ is bordant to $(\theta_{E_1})_{p_1^{\ast}E_2}$, in which $p_1$ is the projection from the sphere bundle $S(E_1\otimes \underline{\mathbb{R}})$ to $M$.
\end{lemma}
\begin{proof}
 Assume the dimension of the fiber of $E_i$ is $n_i$ and write $\theta_{E_1\oplus E_2}$ and $(\theta_{E_1})_{p_1^{\ast}E_2}$ explicitly as $(V, \iota_V, \upsilon_V, \eta_V, E_V)$ and $(W, \iota_W, \upsilon_W, \eta_W, E_W)$ respectively. Then the fibers of $V$ and $W$ are $S^{n_1+n_2}$ and $S^{n_1}\times S^{n_2}$ respectively. We can embed both of the two bundles over $M$ into the vector bundle $E_1\oplus E_2\oplus \mathbb{R}$ as follows. First we choose a Riemannian metric over $E_1\oplus E_2\oplus \mathbb{R}$ and embed $V$ into $E_1\oplus E_2\oplus \mathbb{R}$ as the standard unit sphere bundle. We embed $W$ into $E_1\oplus E_2\oplus \underline{\mathbb{R}}$ such that in each fiber over $p\in M$ it is embedded as follows:
\begin{equation*}
((x,s), (y, t))\mapsto ((5-t)(x, s), y)
\end{equation*}
in which $x\in (E_1)_p$, $y\in (E_2)_p$ and $s,t \in (\underline{\mathbb{R}})_p$. A careful tells us that this indeed induces an embedding of $W$ into $E_1\oplus E_2\oplus \underline{\mathbb{R}}$ and we still denote the image of the embedding by $W$. We embed $V$ into $E_1\oplus E_2\oplus \underline{\mathbb{R}}$ with a scaling such that the radius of each fiber is 10. Denote the standard disk bundle with radius 10 of $E_1\oplus E_1\oplus \underline{\mathbb{R}}$ by $D^{n_1+n_2+1}(10)$ and the solid torus bundle bounds by $V$ by $\bar{V}$. Then we have that $D^{n_1+n_2+1}(10)-\bar{V}$ (which we denote by $Z$) gives rise to a bordism between $V$ and $W$. And $(Z, \iota\circ p_Z, \upsilon_Z, \eta\circ (p_Z\times id), E_Z)$ gives a bordism between $(V, \iota_V, \upsilon_V, \eta_V, E_V)$ and $(W, \iota_W, \upsilon_W, \eta_W, E_W)$.
\end{proof}
\begin{lemma}\label{extension lemma}
Let $\alpha$ be a representable twist over $X$ and $(M, \iota, \upsilon, \eta, E_M)$ be a geometric cycle over $(X, \alpha)$. Let $(\partial M, \iota_{\partial M}, \\
\upsilon_{\partial M}, \eta_{\partial M}, E_{\partial M})$ be the restriction to the boundary of $M$. If a spin$^c$-vector bundle modification with vector bundle $E$ of $(\partial M, \iota_{\partial M}, \upsilon_{\partial M}, \eta_{\partial M}, E_{\partial M})$ is bordant to trivial cycle, then there exists a spin$^c$-vector bundle $V$ over $\partial M$ such that the spin$^c$-vector bundle modification with $V$ is bordant to the trivial cycle and $V$ can be extended to a spin$^c$-vector bundle over $M$.
\end{lemma}
\begin{proof}
Denote the spin$^c$-vector bundle modification of $(\partial M, \iota_{\partial M}, \upsilon_{\partial M}, \eta_{\partial M}, E_{\partial M})$ with vector bundle $E$ by $(Q, \iota_{\partial M}\circ \pi, \upsilon_{Q}, \eta_{Q}, E_{Q})$, which is bordant to the trivial cycle via a bordism of $(W, \iota_W, \upsilon_W, \eta_{W}, E_{W})$. There exists a normal bundle $F$ over $W$, whose restriction to $Q$ is also a normal bundle of $TQ$. On the other hand, by the construction of spin$^c$-vector bundle modification we can observe that there exists a normal bundle of $TQ$ such that it is isomorphic to the pullback (along $\pi$) of the direct sum of a normal bundle of $T\partial M$ (which we denote by $N(T\partial M)$) and a normal bundle of $E$ (which we denote by $N(E)$). Consider the spin$^c$-modification of $(W, \iota_W, \upsilon_W, \eta_W, E_{W})$ with $F\oplus \iota_W^{\ast}(\mathbf{V})$ (here $\mathbf{V}$ is the vector bundle over $X$ with $W_3(V)=[\alpha]$). It gives a bordism from the spin$^c$-modification of $(Q, \iota_{\partial M}\circ \pi, \upsilon_{Q}, \eta_{Q}, E_{Q})$ with $(F\oplus \iota_W^{\ast}(\mathbf{V}))_{|Q}$ and the trivial cycle. According to Lemma \ref{vector bundle modification} and the observation before we can see that the spin$^c$-modification of $(Q, \iota_{\partial M}\circ \pi, \upsilon_{Q}, \eta_{Q}, E_{Q})$ with $(F\oplus \iota_W^{\ast}(\mathbf{V}))_{|Q}$ is bordant to the spin$^c$-modification of $(\partial M, \iota_{\partial M}, \upsilon_{\partial M}, \eta_{\partial M}, E_{\partial M})$ with $E\oplus N(T \partial M)\oplus N(E)\oplus \iota_W^{\ast}(\mathbf{V})_{|Q}$. While $E\oplus N(E)$ is trivial, we get that the spin$^c$-modification of $(\partial M, \iota_{\partial M}, \upsilon_{\partial M}, \eta_{\partial M}, E_{\partial M})$ with $N(T \partial M)\oplus \iota_{\partial M}^{\ast}(\mathbf{V})$ is bordant to a trivial cycle. The normal bundle on the boundary can be extended to the whole manifold obviously and $\iota_{\partial M}^{\ast}(\mathbf{V})$ can be extended to a vector bundle $\iota_M^{\ast}(\mathbf{V})$  So we get our statement.
\end{proof}
Now we start the proof of Theorem \ref{six term exact sequence}.
\begin{proof}
We will show the exactness at $K^g_0(X, \alpha)$, $K^g_0(X, Y; \alpha)$ and $K^g_1(Y, \alpha)$. The proof of the rest part is similarly.
\begin{itemize}
\item  For any $[y]\in K^g_0(Y, \alpha)$, we choose a twisted geometric $K$-cycle $(M_0, \phi_0, \upsilon_0, \\ \eta_0, E_0)$ to represent it. Its image under $j_{\ast}\circ i_{\ast}$ is $(M_0, i\circ \phi_0, \upsilon_0, \eta_0, E_0)$, which is cobordant to trivial $K$-cycle relative $Y$ in $X$. Therefore we have $j_{\ast}\circ i_{\ast}([y])$ is trivial. Assume that $[x]\in K^g_0(X, \alpha)$ and $j^{\ast}([x])=0$. We still choose a twisted geometric $K$-cycle $(M_1, \phi_1, \upsilon_1, \eta_1, E_1)$ to represent $[x]$. Since $j^{\ast}([x])$ is trivial, we obtain that if we do several times of spin$^c$-vector bundle modification for $(M_1, j\circ \phi_1, \upsilon_1, \eta_1, E_1)$ relative to $Y$ in $X$ i.e if we denote the results of $Spin^c$-vector bundle modifications by $(\hat{M_1}, j\circ \phi_1\circ \rho, \upsilon_1', \eta_1', E_1')$, then $(\hat{M_1}, j\circ \phi_1\circ \rho, \upsilon_1', \eta_1', E_1')$ satisfies that $(j\circ \phi_1\circ \rho(\hat{M_1}))\subset Y$, which also implies that $j\circ \phi_1(M)\subset Y$. Therefore we get that $[x]\in$ im $i_{\ast}$.
\item First of all we need to point out that $\delta$ is well defined i.e it is compatible with disjoint union, bordism and spin$^c$-vector bundle modification. It is a tedious check from the definition of $\delta$, which we skip here.  By the definition of $\delta$ we can see that $\delta\circ j_{\ast}=0$. To show that $ker \delta\subset im j_{\ast}$, we choose a twisted geometric $K$-cycle $(M_2, \phi_2, \upsilon_2, \eta_2, E_2)$ such that $\delta[(M_2, \phi_2, \upsilon_2, \eta_2, E_2)]$ is trivial i.e several times of spin$^c$-vector bundle modification for $(\partial M_2, \phi|_{\partial M_2}, \upsilon|_{\partial M_2}, \eta|_{\partial M_2}, E|_{\partial M_2})$ is cobordant to trivial $K$-cycle over $Y$. By Lemma \ref{extension lemma} we can assume that each spin$^c$-vector bundle over the boundary of a manifold can be extended to a spin$^c$-vector bundle over the whole manifold, therefore we get that if we do the spin$^c$-vector bundle modifications for $(M_2, \phi_2, \upsilon_2, \eta_2, E_2)$, then the resulting twisted spin$^c$-manifold is bordant to a twisted spin$^c$-manifold without boundary over $X$. Finally we obtain that $(M_2, \phi_2, \upsilon_2, \eta_2, E_2)$ is equivalent to a twisted geometric $K$-cycle whose underling twisted spin$^c$-manifold is closed, which implies that $[(M_2, \phi_2, \upsilon_2, \eta_2, \\
  E_2)]$ lies in the image of $j_{\ast}$.
\item Let $(M_3, \phi_3, \upsilon_3, \eta_3, E_3)$ be a geometric $\alpha$-twisted $K$-cycle over $(X, Y; \alpha)$. Then $[(\partial M_3, \phi_3|_{\partial M_3}, \upsilon_3|_{\partial M_3}, \eta_3|_{\partial M_3}, E_3|_{\partial M_3})$ is clearly bordant to a trivial twisted geometric $K$-cycle over $X$ i.e $i_{\ast}\circ \delta=0$. Let $[(M_4, \phi_4, \upsilon_4, \eta_4, E_4)] \in K^g_1(Y, \alpha\circ i)$ be a class which lies in the kernel of  $i_{\ast}$. A similar strategy leads us to get that the underling twisted spin$^c$-manifold of several times of spin$^c$-vector bundle modification of $[(M_4, \phi_4, \upsilon_4, \eta_4, E_4)]$ is a boundary of a twisted spin$^c$-manifold over $X$ , from which we can easily get that $[(M_4, \phi_4, \upsilon_4, \eta_4, E_4)]$ belongs to the image of $\delta$.
\end{itemize}
\end{proof}
\begin{remark}
The condition of representability of the twist is essential for the proof here. In general, a twist is not always representable. We leave the six-term exact sequence of geometric twisted $K$-homology for general twists as a further question to be investigated.
\end{remark}
\begin{theorem}[\textbf{Additivity}]
Let $(X_i)_{i\in I}$ be a family of locally finite $CW$-complexes and $\alpha_i: X_i\rightarrow K(\mathbb{Z}, 3)$ be a twist over $X_i$ for each $i$. Moreover, we require that there exists an $\alpha_i$-twisted spin$^c$-manifold over $X_i$ for each $i$. Denote $X$ to be the disjoint union of $X_i$ and $\alpha$ is a twist over $X$ such that the restriction of $X$ to each $X_i$ is $\alpha_i$. Then we have the following isomorphism:
\begin{equation}
K^g_{\ast}(X, \alpha)\cong \oplus_i K^g_{\ast}(X_i, \alpha_i)
\end{equation}
\end{theorem}
The proof is not difficult from the definition, so we skip it here for simplicity.
The above theorem implies the Mayer-Vietoris sequence and the Milnor's $\underleftarrow{lim}^1$-exact sequence of geometric twisted $K$-homology.
\begin{theorem}[\textbf{Mayer-Vietoris sequence}]\label{MV sequence}
Assume two open set $U$ and $V$ of $X$ satisfies $X=U\bigcup V$ and the twist $\alpha$ is representable, we have the Mayer-Vietoris sequence of twisted $K$-homology:
\begin{equation*}
\begin{tikzpicture}
[shorten >=1pt,node distance=1.1cm,auto]
\node[]  (X_0)  {$K_1^{g}(X,\alpha)$};
\node[]  (X_1)  [right= of X_0] {$K_0^{g}(U\bigcap V, \alpha_{U\bigcap V})$};
\node[]  (X_2)[right=of X_1] {$K_0^{g}(U,\alpha_U)\oplus K_0^{g}(V,\alpha_V)$};
\node[]  (Y_0)[below=of X_0] {$K_1^{g}(U,\alpha_U)\oplus K_1^{g}(V,\alpha_V)$};
\node[]  (Y_1) [right=of Y_0] {$K_1^{g}(U\bigcap V, \alpha_{U \bigcap V})$};
\node[]  (Y_2) [right=of Y_1]  {$K_0^{g}(X,\alpha)$};
\path[->]
(X_0) edge[above] node {$\delta$} (X_1)
(X_1) edge[above] node {${j_U}_{\ast}\oplus {j_V}_{\ast}$}  (X_2)
(X_2) edge[right] node {${i_U}_{\ast}-{i_V}_{\ast}$} (Y_2)
(Y_2) edge[below] node {$\delta$} (Y_1)
(Y_1) edge[below] node {${j_U}_{\ast} \oplus {j_V}_{\ast}$}  (Y_0)
(Y_0) edge[left]  node {${i_U}_{\ast}-{i_V}_{\ast}$} (X_0);
\end{tikzpicture}
\end{equation*}
\end{theorem}
\begin{proof}
Let $Z$ be the disjoint union of $U$ and $V$, $Y$ be $U\cap V$. Then consider the six-term exact sequence for the pair $(Z, Y)$ and use the excision isomorphism $K^g_{\ast}(Z, Y; \alpha)\cong K^g_{\ast}(X, \alpha)$ we can get the Mayer-Vietoris sequence for twisted geometric $K$-homology groups.
\end{proof}
\begin{theorem}[$\underleftarrow{lim}^1$-\textbf{exact sequence}]\label{Milnor sequence}
Given a countable $CW$-complex $X$ and a representable twist $\alpha$ over $X$ such that there exists an $\alpha$-twisted spin$^c$-manifold over $X$. We denote the $n$-skeleton of $X$ to be $X_n$ and the inclusion $X_n\hookrightarrow X$ to be $i_n$. Let $\alpha_n= \alpha\circ i_n$. Then we have the following exact sequence
\begin{equation}
1\rightarrow \underleftarrow{lim}^1 K^g_{\ast-1}(X_n, \alpha_n)\rightarrow K^g_{\ast}(X, \alpha)\rightarrow \underleftarrow{lim}K^g_{\ast}(X_n, \alpha_n)\rightarrow 1
\end{equation}
\end{theorem}
The proof is standard and the same as the proof of Milnor's $\underleftarrow{lim}^1$-sequence of homology theory, which can be found in chapter 19 of \cite{May}.

\section{The charge map is an isomorphism}
In Section $8$ of \cite{BCW} they propose a question:
\begin{center}
Is the charge map $h$ (see \eqref{charge map}) an isomorphism?
\end{center}
In this section we show this is true for countable $CW$-complexes by considering the following diagram.
\begin{equation}\label{diagram h}
\begin{tikzpicture}
               [
                    execute at begin node = \(,
                    execute at end node = \),
                    inner sep = .3333em,
                  ]
                  \matrix (m) {
                   |(i)|K^{geo}(X, \mathcal{A}) \#  \#  |(l)| K^{top}(X, \mathcal{A})  \\
                   |(j)|K^g(X, \alpha) \#  \#  |(k)|  K^a(X, \alpha)  \\
                  };
                  \path[->]
                  (i)edge node[auto]{F} (j)
                  (i)edge node[auto]{h}(l)
                  (j)edge node[auto]{\mu}(k)
                  (l)edge node[auto]{\eta}(k) ;
              \end{tikzpicture}
\end{equation}
Here $\mu$ is the analytic index map in \cite{BLW} and $\eta: K^{top}_{\ast}(X, \mathcal{A})\rightarrow K^a_{\ast}(X, \mathcal{A})$ is the natural map in \cite{BCW} , which is defined as follows:
\begin{equation}
\eta(M, \phi, \sigma)= \phi_{\ast}(PD(\sigma))
\end{equation}
Moreover, we know that $\mu$ is an isomorphisms for compact manifolds and $\eta$ is an isomorphism for locally finite $CW$-complexes. And we proved that $F$ is an isomorphism for any locally finite $CW$-complexes. If we can show that diagram \eqref{diagram h} is commutative, then we will get that $h$ must also be an isomorphism.
\begin{proposition}
For any locally finite $CW$-complex $X$, the diagram \eqref{diagram h} is commutative.
\end{proposition}
\begin{proof}
If we write the formula of the map in diagram \eqref{diagram h} for a $D$-cycle $(M, E, \iota, S)$ over $(X, \alpha)$ using the notation before, we get
\begin{eqnarray}
\eta\circ h(M, E, \iota, S)=\iota_{\ast}\circ \rho_{\ast}\circ PD\circ \chi\circ s_!([E])\\
\mu\circ F(M, E, \iota, S)= \iota_{\ast}\circ \eta_{\ast}\circ I_{\ast}\circ PD([E])
\end{eqnarray}
Therefore the commutativity of diagram \eqref{diagram h} is equivalent to
\begin{equation}
\iota_{\ast}\circ \rho_{\ast}\circ PD\circ \chi\circ s_{!}= \iota_{\ast}\circ \eta_{\ast}\circ I_{\ast}\circ PD
\end{equation}
i.e equivalent to the commutativity of the following diagram:
\begin{equation}\label{diagram h1}
\begin{tikzpicture}
               [
                    execute at begin node = \(,
                    execute at end node = \),
                    inner sep = .3333em,
                  ]
                  \matrix (m) {
                   |(i)|K^0(E) \#  \#  |(l)| K_{\ast}(M, W_3\circ \tau)  \\
                   |(j)|K^{\ast}(\hat{M}, W_3\circ \upsilon \circ \varrho)) \#  \#  |(k)|  K_{\ast}(M, W_3\circ \upsilon)  \\
                   |(p)|K^{\ast}(\hat{M}, -\alpha\circ \iota\circ \rho) \#  \#  |(q)| \\
                   |(s)|K_{\ast}(\hat{M}, \alpha\circ \iota\circ \rho) \#  \# |(t)| K_{\ast}(M, \alpha\circ \iota)\\
                  };
                  \path[->]
                  (i)edge node[auto]{s_!}(j)
                  (i)edge node[auto]{PD}(l)
                  (j)edge node[auto]{\chi}(p)
                  (p)edge node[auto]{PD}(s)
                  (l)edge node[auto]{I_{\ast}} (k)
                  (k)edge node[auto]{\eta_{\ast}} (t)
                  (s)edge node[auto]{\rho_{\ast}} (t) ;
              \end{tikzpicture}
\end{equation}
Here $PD: K^{\ast}(\hat{M}, (\iota\circ \rho)^{\ast}(\mathcal{A}^{op})\rightarrow K_{\ast}(\hat{M}, (\iota\circ \rho)(\mathcal{A}))$ is the $Poincar\acute{e}$ duality map between twisted $K$-theory and twisted $K$-homology in \cite{EEK} and \cite{Tu}.
By the naturality of Poincar$\check{e}$ duality (one can see Corollary $3.8$ in \cite{CS}), we have that the commutative diagram below
\begin{equation}\label{diagram h3}
\begin{tikzpicture}
               [
                    execute at begin node = \(,
                    execute at end node = \),
                    inner sep = .3333em,
                  ]
                  \matrix (m) {
                   |(i)|K^0(M) \#  \#  |(l)| K_{\ast}(M, W_3\circ \tau)  \\
                   |(j)|K^0(\hat{M}, W_3\circ \upsilon\circ \rho) \#  \#  |(k)|  K_{\ast}(\hat{M}, W_3\circ \tau\circ \rho)  \\
                  };
                  \draw[transform canvas={xshift=0.5ex},->] (l) edge node[right]{s_{\ast}}(k);
                  \draw[transform canvas={xshift=-0.5ex},->](k) edge node[left]{\rho_{\ast}} (l);
                  \path[->]
                  (i)edge node[auto]{s_!} (j)
                  (i)edge node[auto]{PD}(l)
                  (j)edge node[auto]{PD}(k) ;
              \end{tikzpicture}
\end{equation}
The twist of the lower right $K$-homology group is $W_3\circ \tau\circ \rho$ since $\hat{M}$ admits a spin$^c$-structure.
Since $\rho\circ s=Id$, we have $\rho_{\ast}\circ s_{\ast}=id$, therefore we can get
\begin{equation}\label{diagram equation h1}
PD=\rho_{\ast}\circ s_{\ast}\circ PD= \rho_{\ast}\circ PD\circ s_!
\end{equation}
To prove the whole proposition, we first give the following lemma.
\begin{lemma}\label{commutative lemma}
Denote the map on twisted $K$-homology groups induced by changing twist by $\tilde{\chi}: K_{\ast}(\hat{M}, W_3\circ \tau\circ \rho)\rightarrow K_{\ast}(\hat{M}, \alpha\circ \iota\circ \rho)$. Then the following diagram is commutative
\begin{equation}\label{diagram h2}
\begin{tikzpicture}
               [
                    execute at begin node = \(,
                    execute at end node = \),
                    inner sep = .3333em,
                  ]
                  \matrix (m) {
                   |(i)|K^{\ast}(\hat{M}, W_3\circ \upsilon\circ \rho) \#   |(l)| K_{\ast}(\hat{M}, W_3\circ \tau\circ \rho) \# |(p)|K_{\ast}(M, W_3\circ \tau)  \\
                   |(j)|K^{\ast}(\hat{M}, -\alpha\circ \iota\circ \rho) \#   |(k)| K_{\ast}(\hat{M}, \alpha\circ \iota\circ\rho) \#  |(q)|K_{\ast}(M, \alpha\circ \iota)  \\
                  };
                  \path[->]
                  (i)edge node[auto]{\chi} (j)
                  (i)edge node[auto]{PD}(l)
                  (j)edge node[auto]{PD}(k)
                  (l)edge node[auto]{\tilde{\chi}}(k)
                  (l)edge node[auto]{\rho_{\ast}} (p)
                  (k)edge node[auto]{\rho_{\ast}} (q)
                  (p)edge node[auto]{\eta_{\ast}\circ I_{\ast}}(q) ;
              \end{tikzpicture}
\end{equation}
\end{lemma}
If we combine the above lemma and diagram \eqref{diagram h3}, we can get that diagram \eqref{diagram h} is commutative.
\end{proof}
\begin{corollary}
If $X$ is an smooth manifold and $\mathcal{A}$ is a twisting on $X$, then the charge map $h: K^{geo}_{\ast}(X, \alpha)\rightarrow K^{top}_{\ast}(X, \alpha)$ is an isomorphism.
\end{corollary}
\begin{proof}[Proof of \eqref{commutative lemma}]
\par
\begin{itemize}
  \item \textbf{Step 1}
  We prove the first square in diagram (\eqref{diagram h2}) is commutative. First of all, we review the definition of Poincar$\acute{e}$ duality $PD: K^{\ast}(\hat{M}, \mathcal{A})\rightarrow K_{\ast}(\hat{M}, \mathcal{A}^{op})$ for twisted $K$-theory in Lemma $2.1$ of \cite{EEK}:
  \begin{equation}\label{Poincare duality}
  PD(x)=\sigma_{C(\hat{M}, \mathcal{A}^{op})}(x)
  \end{equation}
  Here $\sigma_{C(\hat{M}, \mathcal{A}^{op})}: KK(C(\hat{M}, \mathcal{A})\hat{\otimes}A, B)\rightarrow KK(A, C(\hat{M}, \mathcal{A}^{op})\hat{\otimes}B)$ is a canonical isomorphism for any $C^{\ast}$-algebra $A$ and $B$, which is given by tensoring with $C(\hat{M}, \mathcal{A}^{op})$. If we choose $A$ and $B$ both to be $\mathbb{C}$ and $C(\hat{M}, \mathcal{A})$ to be $C(\hat{M}, (W_3\circ \upsilon\circ \rho)^{\ast}(\mathfrak{K}))$, then we get the Poincar$\acute{e}$ duality $PD$ on the top of the first square . If we choose $A$ and $B$ to be $\mathbb{C}$ and $C(\hat{M}, \mathcal{A})$ to be $C(\hat{M}, (-\alpha\circ \iota\circ \rho)^{\ast}(\mathfrak{K}))$, then we get the Poincar$\acute{e}$ duality $PD$ on the bottom of the first square. Then the commutativity of the first square follows from that $PD$ is natural over $C(\hat{M}, \mathcal{A}^{op})$.
  \item \textbf{Step 2}
  From the definition of $\tilde{\chi}$ , we know that it is induced by changing the twist from $W_3\circ \tau\circ \rho$ to $\alpha\circ \iota\circ \rho$ using the trivialization given by the spinor bundle $\rho^{\ast}S$ and the canonical trivialization of $TM\oplus \upsilon$. On the other hand, we know that $I_{\ast}$ is the changing twist map induced by the canonical trivialization of $TM\oplus \upsilon$ and $\eta_{\ast}$ is the map induced by the homotopy $\eta$, which is induced by the spinor bundle $S$. So the commutativity of the second square follows from that the changing twist map is natural.
\end{itemize}
\end{proof}
In section 4, we proved that the geometric twisted $K$-homology group defined in \cite{BCW} is homotopy invariant and each finite $CW$-complex is homotopy equivalent to a smooth manifold. Therefore we have that the charge map is an isomorphism for any finite $CW$-complex.

\section{Geometric twisted K-cycles via bundle gerbes}
In this section we give another definition of geometric twisted $K$-homology via bundle gerbes. First of all, we give the definition of bundle gerbes first.
\begin{definition}
Given a $CW$-complex $B$, a bundle gerbe over $B$ is a pair $(P, Y)$, where $\pi: Y\rightarrow B$ is a locally split map and $P$ is a principal $U(1)$-bundle over $Y\times_M Y$ with an associative product, i.e for every point $(y_1, y_2), (y_2, y_3)\in Y\times_M Y$, there is an isomorphism
\begin{equation}
P_{(y_1, y_2)}\otimes_{\mathbb{C}} P_{(y_2, y_3)}\rightarrow P_{(y_1, y_3)}
\end{equation}
and the following diagram commutes
\begin{equation}
\begin{tikzpicture}
[
                    execute at begin node = \(,
                    execute at end node = \),
                    inner sep = .3333em,
                  ]
                  \matrix (m) {
                   |(i)|P_{(y_1, y_2)}\otimes P_{(y_2, y_3)}\otimes P_{(y_3, y_4)} \# \# |(l)|P_{(y_1, y_3)}\otimes P_{(y_3, y_4)}  \\
                   |(j)|P_{(y_1, y_2)}\otimes P_{(y_2, y_4)}  \#  \#  |(k)|P_{(y_1, y_4)} \\
                  };
                  \path[->]
                  (i)edge node[auto,swap]{} (j) edge node[auto]{} (l)
                  (j)edge node[auto,swap]{}(k)
                  (l)edge node[auto]{}(k) ;
\end{tikzpicture}
\end{equation}
\end{definition}
\par
For every principal $U(1)$-bundle $J$ over $B$ we can define a bundle gerbe $\delta(J)$ by $\delta(J)_{(y_1, y_2)}=J_{y_1}\otimes J^{\ast}_{y_2}$. The product is induced by the pairing between $J^{\ast}$ and $J$. A bundle gerbe $(P, Y)$ is called trivial if there is a hermitian line bundle $J$ such that $P\cong \delta(J)$. In this case, $J$ is called a trivialization of $(P, Y)$. For every bundle gerbe $(P, Y)$ over $M$ we can associate a third integer cohomology class $d(P)\in H^3(M, \mathbb{Z})$ to describe the non-triviality of the bundle gerbe which we call Dixmier-Douady class (see \cite{Murray}).
\begin{definition}
Two bundle gerbes $(P, Y)$ and $(Q, Z)$ are stable isomorphic to each other if there is a trivialization of $p_1^{\ast}(P)\otimes p_2^{\ast}(Q)^{\ast}$. Here $p_1: Y\times_B Z\rightarrow Y$ and $p_2: Y\times_B Z\rightarrow Z$ are the natural projections. And the trivialization is called a stable isomorphism between $(P, Y)$ and $(Q, Z)$.
\end{definition}
\par
The following theorem gives the relation between stable isomorphism classes and Dixmier-Douady classes.
\begin{theorem}
Two bundle gerbes are stable isomorphic iff their Dixmier-Douady classes are equal to each other. Moreover, the Dixmier-Douady class defines a bijection between stable isomorphic classes of bundle gerbes over $M$ and $H^3(M, \mathbb{Z})$.
\end{theorem}
\par
Now we give another definition which is important in our construction of geometric twisted $K$-homology.
\begin{definition}
Let $(P, Y)$ be a bundle gerbe over $B$. A finite dimensional hermitian bundle $E$ over $Y$ is called a $(P, Y)$-module if there is a complex vector bundle isomorphism
\begin{equation*}
\phi: P\otimes \pi_1^{\ast}(E)\cong \pi_2^{\ast}(E)
\end{equation*}
which is compatible with the bundle gerbe product, i.e the following diagram is commutative
\begin{equation*}
\begin{tikzpicture}
[
                    execute at begin node = \(,
                    execute at end node = \),
                    inner sep = .3333em,
                  ]
                  \matrix (m) {
                   |(i)|P_{(y_1, y_2)}\otimes P_{(y_2, y_3)}\otimes E_{y_3} \# \# |(l)|P_{(y_1, y_3)}\otimes E_{y_3}  \\
                   |(j)|P_{(y_1, y_2)}\otimes E_{y_2} \#  \#  |(k)|E_{y_1} \\
                  };
                  \path[->]
                  (i)edge node[auto,swap]{} (j) edge node[auto]{} (l)
                  (j)edge node[auto,swap]{}(k)
                  (l)edge node[auto]{}(k) ;
\end{tikzpicture}
\end{equation*}
where $\pi_i$ ($i=1, 2$) are two projections from $Y\times_B Y$ to $Y$. Moreover, the Grothendieck group of isomorphism classes of bundle gerbe module over $(P, Y)$ is called the $K$-group of $(P, Y)$.
\end{definition}
Now we give the construction of geometric twisted $K$-cycles
\begin{definition}
Let $B$ be a space , $H\in H^3_{tor}(X, \mathbb{Z})$ and $(P, Y)$ be a bundle gerbe over $B$ with $-H$ as Dixmier-Douady class. A geometric twisted $K$-cycle is a triple $(M, f, E)$ where
\begin{itemize}
\item $M$ is a compact spin$^c$-manifold;
\item $f: M\rightarrow B$ is continuous;
\item $[E]$ is an isomorphism class of $f^{\ast}(P, Y)$-module.
\end{itemize}
We denote the whole twisted $K$-cycles over $(B, H)$  by $\Gamma_{(P, Y)}(B)$
\end{definition}
To give twisted $K$-homology group, we also need to define an equivalence $\sim$ on these geometric cycles as follows
\begin{itemize}
\item{\textbf{Direct sum-disjoint union}} For any two cycles $(M, f, E_1)$ and $(M, f, E_2)$ over $(B, (P, Y))$, then we have
 \begin{equation}
 (M, f, [E_1])\cup (M, f, [E_2])\sim (M, f, [E_1]+[E_2])
 \end{equation}
 \item{\textbf{Bordism}} Given two cycles $(M_0, f_0, E_0)$ and $(M_1, f_1, E_1)$ over $(B, (P, Y))$, if there exists a cycle $(M, f, E)$ over $(B, (P, Y))$ such that
     \begin{equation}
     \partial M = -M_0\cup M_1
     \end{equation}
     and $E_{\partial M}= -[E_0]\cup [E_1]$, then
     \begin{equation}
     (M_0, f_0, E_0)\sim (M_1, f_1, E_1)
     \end{equation}
 \item{\textbf{Spin$^c$-vector bundle modification}} Let $(M, f, E)$ be a geometric twisted $K$-cycle over $(B, (P, Y))$ and an even dimensional spin$^c$-vector bundle $V$ over $M$. Let $\hat{M}$ to be the sphere bundle of $V\oplus \underline{\mathbb{R}}$. Denote the bundle map by $\rho: \hat{M}\rightarrow M$ and the positive spinor bundle of $T^v(\hat{M})$ by $S_V^+$. The vector bundle $S_V^+\otimes \rho^{\ast}(E)$ over $\hat{M}$ is a $(\rho\circ f)^{\ast}(P, Y)$ module. Then
     \begin{equation}
     (M, f,[ E])\sim (\hat{M}, \rho\circ f, [S_V^+\otimes \rho^{\ast}(E)])
     \end{equation}
\end{itemize}
\begin{definition}
For any space $B$ and bundle gerbe $(P, H)$ over $B$. We define $K^{gg}_{i, (P, Y)}(B, H)= \Gamma_i(B, (P, Y))/ \sim$ ($i$=$0$, $1$). The parity depends on the dimension of the spin$^c$-manifold in a twisted $K$-cycle.
\end{definition}
\begin{proposition}
If $(P, Y)$ and $(Q, Z)$ are two bundle gerbes over $B$ with the same Dixmier-Douady class $-H$, then we have $K^{gg}_{i, (P, Y)}(B, H)$ is isomorphic to $K^{gg}_{i, (Q, Z)}(B, H)$.
\end{proposition}
\begin{proof}
Let $R$ be a stable isomorphism between $(P, Y)$ and $(Q, Z)$ i.e  a trivialization of $p_1^{\ast}(P)\otimes p_2^{\ast}(Q)$. Without loss of generality we can just assume that $Z=Y$. Otherwise we can consider the bundle gerbe $(p_1^{\ast}P, Y\times_B Z)$ and $(p_2^{\ast}Q, Y\times_B Z)$ instead.  Let $(M, f, [E])\in \Gamma^i(B, (P, Y))$. Since $Q\cong P\otimes R$, therefore $(M, f, [E]\otimes L_{R})$ (here $L_{R}$ is the natural associated line bundle of $R$) is a twisted geometric $K$-cycle over $(B, (Q, Y))$. So we get a homomorphism from $\Gamma(B, (P, Y))$ to $\Gamma(B, (Q, Z))$, which we denote by $r$. A tedious check tells us that $r$ respects disjoint union, bordism and spin$^c$-bundle modification. Therefore $r$ induces a homomorphism from $K_i^{gg}(B, (P, Y))$ to $K_i^{gg}(B, (Q, Y))$. If we change the roles of $(P, Y)$ and $(Q, Z)$ in the above construction, then we get an inverse of $r$. Therefore $r$ is an isomorphism.
\end{proof}
Let $(P, Y)$ be a bundle gerbe over $B$ with Dixmier-Douady class $H$. According to \textbf{Proposition 6.4} in \cite{BCMMS}, the $i$-th $K$-group of bundle gerbe $(P, Y)$ is isomorphic to the $i$-th twisted $K$-group $K^i(X, -H)$. Then the definitions of $K^{gg}_{\ast}(X, H)$ and $K^{top}_{\ast}(X, \alpha)$ implies the following proposition:
\begin{proposition}
Let $X$ be a finite $CW$-complex and $H\in H^3_{torsion}(X, \mathbb{Z})$. Then we have
\begin{equation}
K^{gg}_i(X, (P, Y))\cong K^{top}_i(X, \alpha)\cong K^{geo}_i(X, \alpha)
\end{equation}
\end{proposition}

\section{Some constructions about geometric twisted K-cycles}
First of all we review notions about topological $T$-duality in \cite{TB}.
Let $\pi: P\rightarrow B$ and $\hat{\pi}: \hat{P}\rightarrow B$ be principal $S^1$-bundles over $B$ and $H\in H^3(P, \mathbb{Z}),  \hat{H}\in H^3(\hat{P}, \mathbb{Z})$. Denote the associated line bundles of $P$ and $\hat{P}$ by $E$ and $\hat{E}$ respectively. Let $V= E\bigoplus \hat{E}$ and $r: S(V)\rightarrow B$ be the unit sphere bundle of $V$.
\begin{definition}
A class $\mathbf{Th}\in H^3(S(V), \mathbb{Z})$ is called Thom class if $r_!(\mathbf{Th})=1 \in H^0(B, \mathbb{Z})$.
\end{definition}
Let $i: P\rightarrow S(V)$ and $\hat{i}: \hat{P}\rightarrow S(V)$ be inclusion of principal $S^1$-bundle into $S^3$-bundle.
\begin{definition}
We say that $((P, H), (\hat{P}, \hat{H}))$ is a $T$-dual pair over $B$ or $(P, H)$ and $(\hat{P}, \hat{H})$ are $T$-dual to each other if there exists a Thom class $\mathbf{Th}\in H^3(S(V),\mathbb{Z})$ such that
\begin{equation}
H=i^{\ast}(\mathbf{Th}),\ \  \hat{H}=\hat{i}^{\ast}(\mathbf{Th})
\end{equation}
\end{definition}
Given a $T$-dual pair $((P, H), (\hat{P}, \hat{H}))$ over $B$, there exists a $T$-duality isomorphism between the associated twisted $K$-groups, which is given by
\begin{equation}
T= \hat{p}_!\circ u \circ p^{\ast}: K^i(P, H)\rightarrow K^{i-1}(\hat{P}, \hat{H})
\end{equation}
Here $p^{\ast}$ and $\hat{p}_!$ are the pullback and pushforward maps respectively. $u$ is a changing twist map from $K^i(P\times_B\hat{P}, p^{\ast}(H))$ to $K^i(P\times_B\hat{P}, \hat{p}^{\ast}(\hat{H}))$.
In order to establish the $T$-duality transformation of geometric twisted $K$-homology we first give the construction of analogue maps of induced map , wrong way map and changing twist map for twisted $K$-homology.
\begin{enumerate}
\item{\textbf{Induced map}}
 Assume $f:(X_1,Y_1)\rightarrow (X_2,Y_2)$ is a continuous map between two pairs of topological spaces. Then the induced map $f_{\ast}: K^g_i(X_1, Y_1; \alpha\circ f)\rightarrow K^g_i(X_2, Y_2; \alpha)$ is defined as follows:
    \begin{equation}
    f_{\ast}([M,\phi,\upsilon,\eta, E])=[(M,\phi \circ f,\upsilon,\eta, E)]
    \end{equation}
\item{\textbf{Wrong way map}} Let $f: P\rightarrow N$ be a $K$-oriented map between smooth manifolds and $\alpha: N\rightarrow K(\mathbb{Z},3)$ be a twist over $N$. we define the wrong way map $f^!: K^g_{i-1}(N, \alpha)\rightarrow K^g _i(P, \alpha\circ f)$ to be
  \begin{equation}
  \pi^!([M,\phi,\upsilon,\eta, E])=[(\tilde{M},\tilde{\phi},\tilde{\upsilon},\tilde{\eta},\tilde{\pi}^{\ast}(E))]
  \end{equation}
Here $\tilde{M}$ is the fiber product $M\times_N M$, $\tilde{\upsilon}$ is the stable normal bundle of $\tilde{M}$ and $\tilde{\eta}$ is induced by $\eta$.
\begin{remark}
As $f$ is $K$-oriented, therefore $W_3(\tilde{\upsilon}\oplus f^{\ast}(\upsilon))$ is trivial, which implies that $W_3\circ \tilde{\upsilon}$ is homotopic to $\alpha\circ f\circ \tilde{\phi}$ via a homotopy $\lambda$. $\tilde{\eta}: \tilde{M}\times [0, 1]\rightarrow K(\mathbb{Z}, 3)$ is given by the combination of $\lambda$ and $\eta$ as follows:
\begin{eqnarray*}
\tilde{\eta}(x, t)=
\begin{cases}
  \lambda(x, 2t), & 0\leq t\leq 1/2;\\
  \eta\circ (f'\times id)(x, 2t-1), & 1/2\leq t\leq 1.
\end{cases}
\end{eqnarray*}
in which $f': \tilde{M}\rightarrow M$ is the canonical projection to $M$. Therefore we get that $(\tilde{M},\tilde{\phi},\tilde{\upsilon},\tilde{\eta},\tilde{\pi}^{\ast}(E))$ is a twisted geometric cycle over $(P, \alpha\circ f)$. In particular, when $f$ is the bundle map for a principal $S^1$-bundle, $\tilde{M}$ is the pullback $S^1$-bundle along $f$.
\end{remark}
\item{\textbf{Changing twist map}} Let $P\rightarrow B$ be a principal $\mathbb{T}^2$-bundle over $B$. If we have two twist $\alpha_1,\alpha_2:P \rightarrow K(\mathbb{Z},3)$ and their associated cohomology classes are the same, then $\alpha_1$ and $\alpha_2$ are homotopic. Choose a  homotopy $h$ such that the restriction of $h$ to each fiber of $P$ corresponds to the cohomology class $\theta\cup \hat{\theta}\in H^2(P_b, \mathbb{Z})$. Here $\theta $ and $\hat{\theta}$ are generators of the first cohomology group of the two copies of $S^1$ of a fiber. Then for a geometric cycle $\delta$ of $(X,\alpha_1)$ one can define the changing twist map $u: K^g_i(P, \alpha_1)\rightarrow K^g_i(P, \alpha_2)$ as follows:
    \begin{equation}
    u([M,\phi,\upsilon,\eta, E])=[(M,\phi,\upsilon,\hat{\eta},E)]
    \end{equation}
    Here $\hat{\eta}$ is induced by the following diagram:
    \begin{equation*}
\begin{tikzpicture}
[shorten >=1pt,node distance=1.5cm,auto]
\node []  (M)  {$W_3\circ \upsilon$};
\node[]  (X)[right=of M] {$\alpha_1 \circ \phi$};
\node[]  (Y)[right=of X] {$\alpha_2\circ \phi$};
\path[->]
(M)  edge[above] node {$\eta$} (X)
(X) edge[above] node {$h\circ(\phi\times id)$} (Y);
\end{tikzpicture}
\end{equation*}
More explicitly, the $\hat{\eta}$ is given by the composition of $\eta$ and $h\circ (\phi\times id)$, which we denote by $ (h\circ (\phi\times id))\ast \eta$
\begin{eqnarray*}
 (h\circ (\phi\times id))(\eta)(x, t)=
\begin{cases}
  \eta(x, 2t), & 0\leq t\leq 1/2;\\
  h\circ(\phi\times id)(x, 2t-1), & 1/2\leq t\leq 1.
\end{cases}
\end{eqnarray*}
\end{enumerate}
\begin{lemma}
The induced map, wrong way map and changing twist map above are all compatible with disjoint union, bordism and spin$^c$-vector bundle modification.
\end{lemma}
\begin{proof}
We have proved the induced map part in Lemma \ref{induced map} and it is not hard to see that they all respect the disjoint union. We do the rest here.
\begin{itemize}
\item  Let $(M, \phi, \upsilon, \eta, E)$ be a bordism between $(M_1, \phi_1, \upsilon_1, \eta_1, E_1)$ and $(M_2, \phi_2,$ \\
$\upsilon_2, \eta_2, E_2)$ over $X$. Denote $p^!(M_i, \phi_i, \upsilon_i, \eta_i, E_i)$ by $(\tilde{M_i}, \tilde{\phi_i}, \tilde{\upsilon_i}, \tilde{\eta_i}, \tilde{\pi}^{\ast}E_i)$. Since the boundary of a pullback space is the pullback of the original boundary and the stable normal bundle of a principal $S^1$-bundle is isomorphic to the stable normal bundle of the base space, we get that $(\tilde{M}, \tilde{\phi}, \\ \tilde{\upsilon}, \tilde{\eta}, \tilde{\pi}^{\ast}E)$ gives a bordism between
$(\tilde{M_1}, \tilde{\phi_1}, \tilde{\upsilon_1}, \tilde{\eta_1}, \tilde{\pi}^{\ast}E_1)$ and $(\tilde{M_2}, \tilde{\phi_2},\\
\tilde{\upsilon_2}, \tilde{\eta_2}, \tilde{\pi}^{\ast}E_2)$.

Use the notation above. $p^!(\hat{M}, \phi\circ \rho, \upsilon', \eta', S^+_V\otimes \rho^{\ast}E)$ is $(\tilde{\hat{M}}, \phi\circ \rho \circ \tilde{\hat{\pi}}, \tilde{\upsilon'}, \tilde{\eta'},$ \\
$\tilde{\pi}^{\ast}(S^+_V\otimes \rho^{\ast}E))$. On the other hand,  $\tilde{\pi}^{\ast}V$ is a spin$^c$-vector bundle over $\tilde{M}$. The associated spin$^c$-vector bundle modification of $(\tilde{M}, \tilde{\phi}, \tilde{\upsilon}, \tilde{\eta}, \tilde{\pi}^{\ast}E)$ is that $(\tilde{\hat{M}}, \tilde{\phi}\circ \tilde{\rho}, \upsilon'', \tilde{\eta}', \tilde{\pi}^{\ast}(S^+_V\otimes \rho^{\ast}E))$. The maps in the above twisted geometric $K$-cycles are shown in the following diagram:
\begin{equation}
             \begin{tikzpicture}
               [
                    execute at begin node = \(,
                    execute at end node = \),
                    inner sep = .3333em,
                  ]
                  \matrix (m) {
                   \# |(i)|\tilde{\hat{M}} \#  \\
                   |(j)|\tilde{M}  \#  \#  |(k)|\hat{M} \\
                    \# |(l)|M \#  \\
                  };
                  \path[->]
                  (i)edge node[auto,swap]{\tilde{\rho}} (j) edge node[auto]{\tilde{\hat{\pi}}} (k)
                  (j)edge node[auto,swap]{\tilde{\pi}}(l)
                  (k)edge node[auto]{\rho}(l) ;
              \end{tikzpicture}
              \end{equation}
  By the commutativity, we have that $\phi\circ \rho\circ \tilde{\hat{\pi}}= \tilde{\phi}\circ \tilde{\rho}$. Moreover, $\tilde{\upsilon'}$ and $\upsilon''$ are homotopic because they are both classifying maps of the stable normal bundle of $\tilde{\hat{M}}$. The coincidence of $\tilde{\upsilon'}$ and $\upsilon''$ implies that $\tilde{\eta'}$ and $\tilde{\eta}'$ are homotopic to each other. So the wrong way map respects the spin$^c$-vector bundle modification relation.
\item Use the above notions. $u((M,\phi, \upsilon, \eta, E))$ still gives a bordism between $(M_1, \phi_1, \upsilon_1, \hat{\eta_1}, E_1)$ and $(M_2, \phi_2, \upsilon_2, \hat{\eta_2}, E_2)$ i.e $u$ respects bordism equivalence. Since the composition of homotopies are associative up to homotopy, we get that $u$ respects the spin$^c$-vector bundle modification relation.
\end{itemize}
\end{proof}

\section{T-duality for twisted geometric K-homology}
\begin{theorem}\label{T duality for twisted K homology}
Let $B$ be a finite $CW$-complex and $((P, H), (\hat{P}, \hat{H}))$ be a $T$-dual pair over $B$.
\begin{equation*}
             \begin{tikzpicture}
               [
                    execute at begin node = \(,
                    execute at end node = \),
                    inner sep = .3333em,
                  ]
                  \matrix (m) {
                   \# |(i)|P\times_B \hat{P} \#  \\
                   |(j)|P  \#  \#  |(k)|\hat{P} \\
                    \# |(l)|B \#  \\
                  };
                  \path[->]
                  (i)edge node[auto,swap]{p} (j) edge node[auto]{\hat{p}} (k)
                  (j)edge node[auto,swap]{\pi}(l)
                  (k)edge node[auto]{\hat{\pi}}(l) ;
              \end{tikzpicture}
              \end{equation*}
Moreover, we assume that $\alpha: P\rightarrow K(\mathbb{Z}, 3)$ and $\hat{\alpha}: \hat{P}\rightarrow K(\mathbb{Z}, 3)$ satisfy that $\alpha^{\ast}([\Theta])= H$ and $\hat{\alpha^{\ast}}([\Theta])=\hat{H}$ (Here $[\Theta]$ is the positive generator of $H^3(K(\mathbb{Z},3), \mathbb{Z})$). Moreover we assume that both $\alpha$ and $[\hat{\alpha}]$ is representable. Then the map $T=\hat{p}_{\ast}\circ u\circ p^!:K_{\ast}^{g}(P,\alpha)\mapsto K_{\ast+1}^{g}(\hat{P},\hat{\alpha})$ is an isomorphism.
\end{theorem}
To prove the theorem \eqref{T duality for twisted K homology} we need the following lemmas.
\begin{lemma}
$T$ is compatible with the boundary operator and the induced map in the Mayer-Vietoris sequence.
\end{lemma}
\begin{proof}
We first prove the compatibility with the induced map. Assume we have a map $f:X \mapsto Y$ and we have the associated $T$-duality diagrams over $Y$ and pullback it to $X$. $f$ induces maps by $F:P_X\mapsto P_Y$,
$\hat{F}:\hat{P_X} \rightarrow \hat{P_Y}$ and $G: P_X \times_X \hat{P_X} \rightarrow P_Y \times_Y \hat{P_Y}$. Then we have the following identities:
\begin{equation}
\begin{aligned}
\hat{F}_{\ast} \circ T_X&=& \hat{F}_{\ast}\circ (\hat{p}_X)_\ast \circ u_X \circ p_X^! \\
                        &=&(\hat{p}_Y)_{\ast} \circ G_{\ast} \circ u_X \circ p_X^!  \\
                        &=& (\hat{p}_Y)_{\ast} \circ u_Y \circ (G\circ p_Y)^!\circ F_{\ast} \\
                        &=& T_Y \circ F_{\ast}
\end{aligned}
\end{equation}
Now we turn to the compatibility with the boundary map, in the Mayer-Vietoris sequence of the boundary operator$\delta: K^g_{\ast}(X, \alpha)\rightarrow K^g_{\ast+1}(U\cap V, \alpha\circ i_{U\cap V})$ is given as follows:
Choose a continuous map $f: X\rightarrow [0, 1]$ such that $f_{U-U\cap V}$ is $0$ and $f_{V-U\cap V}$ is $1$. Without loss of generality we assume that $f\circ \phi: M\rightarrow [0,1]$ is a smooth function and $1/2$ is a regular point of $f\circ \phi$. For any twisted geometric $K$-cycle $x=(M, \phi, \upsilon, \eta, E)$, define $\delta x= (f^{-1}(1/2), \phi\circ i, \upsilon\circ i, \eta\circ (i\times id), i^{\ast}E)$. By this formula, we get that $\delta$ is compatible with induced map. Also the homotopies $(h\circ (\phi\circ i\times id))\ast (\eta\circ (i\times id))$ and $(\eta\ast (h\circ (\phi\times id))\circ (i\times id)$ are homotopic to each other, which implies that $u \circ\delta = \delta \circ u$. The remainder is to show that $\hat{p}^!\circ \delta = \delta \circ \hat{p}^!$. We write both sides explicitly first:
\par
Given a principal $S^1$-bundle $\pi: P\rightarrow B$ and a twisted geometric cycle $(M, \phi, \upsilon, \eta, E)$ (which we denote by $x$) over $P$
\begin{eqnarray*}
\hat{p}^!\circ \delta x= (\tilde{(f\circ \phi)^{-1}(1/2)}), \phi\circ \tilde{\pi}\circ i, \tilde{\upsilon}\circ i, \eta\circ ((\tilde{\pi}\circ i)\times id), (i\circ \tilde{\pi})^{\ast}E); \\
\delta \circ \hat{p}^!x=((f\circ \phi\circ \tilde{\pi})^{-1}(1/2), \phi\circ \tilde{\pi}\circ i, \tilde{\upsilon}\circ i, \eta\circ ((\tilde{\pi}\circ i)\times id), (i\circ \tilde{\pi})^{\ast} E )
\end{eqnarray*}
Since $\tilde{f\circ \phi}^{-1}(1/2)$ is exactly $(f\circ \phi\circ \tilde{\pi})^{-1}(1/2)$, we get that $\hat{p}^!\circ \delta= \delta\circ \hat{p}^!$. Finally, we have that
\begin{equation*}
T\circ \delta = (\hat{p}^!\circ u\circ p_{\ast})\circ \delta = \delta\circ (\hat{p}^!\circ u\circ p_{\ast})= \delta\circ T
\end{equation*}
\end{proof}
\begin{lemma}
\end{lemma}
\begin{proof}[Proof of Theorem \ref{T duality for twisted K homology}]
 We do the proof by induction on the number of cells. Assume $X$ is a point, then $P$ and $\hat{P}$ are both $S^1$ and the correspondence space is $S^1\times S^1$. Denote $p_{\ast}\circ t^{-1} \circ \hat{p}$ by $T'$. We can see that for any geometric cycle $(M,\iota,\upsilon,\eta,[E])$, the image of this cycle under the map $T' \circ T$ is
$(M\times S^1 \times S^1, \iota\circ \breve{p},\upsilon \circ \breve{p},\breve{\eta},\breve{p}^{\ast}([E]))$.
Here $\breve{p}$ is the projection from $M\times S^1 \times S^1$ to $M$. One can see that this is cobordant to a trivial $S^2$-bundle as spin$^c$-manifolds. The cobordism can be given by $M\times B$, where $B$ is a solid torus with a solid disk cut from it. As all of the bundles involved are trivial, the geometric cycle $(M\times B, \iota \circ \acute{p},\upsilon \circ \acute{p},\acute{\eta},\acute{p}^{\ast}([E]) $ gives the cobordism between the image of $T$ and the spin$^c$-modification i.e $T'\circ T$ is equal to identity in this case. As a consequence we also get that $T$ is an isomorphism. Assume $T$ is an isomorphism when the number of cells is $n$,then we adjoin another cell $\sigma_{n+1}$ to $X$ and we choose open set $U=X\cup \sigma_{n+1}-pt, V= \sigma_{n+1}-\bar{pt}$ and we can get the Mayer-Vietoris sequence of geometric twisted $K$-homology groups. Then the conclusion of this theorem follows from induction and the Five-Lemma.
\end{proof}
\begin{remark}
The construction of $T$-duality transformation of geometric twisted $K$-homology can be easily generalized to $T$-dual pairs of higher dimensional torus bundles.
\end{remark}
We end up with this paper with an interesting question. Can we release the condition in Remark \ref{FW} and construct geometric twisted $K$-theory for general twists?  An idea is to replace twisted spin$^c$-manifolds in the construction of geometric $K$-cycles by noncommutative analogue objects. We leave this for further investigation.

\end{document}